\documentclass[11pt,reqno,final]{amsart}

\usepackage{amsmath, amssymb, epsfig}
\usepackage{mathrsfs}
\usepackage{color}
\usepackage{alg}
\usepackage{url}

\usepackage[font=small,margin=10pt,labelfont={bf},labelsep={space}]{caption}
\usepackage{subfig}
\usepackage{wrapfig}

\usepackage[scaled]{helvet} 
\usepackage{fourier} 
\usepackage{bm}

\makeatletter
\def\algspacing{\alg@unmargin}
\makeatother

\usepackage[margin=1in]{geometry}

\newlength{\algorithmwidth}
\theoremstyle{plain}
\newtheorem{theorem}{Theorem}[section]
\newtheorem{proposition}[theorem]{Proposition}
\newtheorem{corollary}[theorem]{Corollary}

\theoremstyle{definition}
\newtheorem{definition}[theorem]{Definition}

\theoremstyle{remark}

\numberwithin{theorem}{section}
\numberwithin{equation}{section}

\newcommand{\defby}{\overset{\mathrm{\scriptscriptstyle{def}}}{=}}

\def \E {\mathbb{E}}

\def \eps {\varepsilon}

\newcommand{\vct}[1]{\bm{#1}}
\newcommand{\mtx}[1]{\bm{#1}}

\newcommand{\norm}[1]{\lVert {#1} \rVert}
\newcommand{\normsq}[1]{\norm{#1}^2}
\newcommand{\enorm}[1]{\norm{#1}_2}
\newcommand{\enormsq}[1]{\enorm{#1}^2}
\newcommand{\fnorm}[1]{\norm{#1}_{\rm F}}
\newcommand{\fnormsq}[1]{\fnorm{#1}^2}

\newcommand{\Expect}{\E}

\newcommand{\pinv}{\dagger}
\newcommand{\cnst}[1]{\mathrm{#1}}
\newcommand{\Id}{\mathbf{I}}

\begin{document}

\title{Block Kaczmarz Method with Inequalities}
\author{J. Briskman and D. Needell}
\date{\today}

\begin{abstract}
The randomized Kaczmarz method is an iterative algorithm that solves systems of linear equations. Recently, the randomized method was extended to systems of equalities and inequalities by Leventhal and Lewis. Even more recently, Needell and Tropp provided an analysis of a block version of this randomized method for systems of linear equations.  This paper considers the use of a block type method for systems of mixed equalities and inequalities, bridging these two bodies of work.  We show that utilizing a matrix paving over the equalities of the system can lead to significantly improved convergence, and prove a linear convergence rate as in the standard block method.  We also demonstrate that using blocks of inequalities offers similar improvement only when the system satisfies a certain geometric property.  We support the theoretical analysis with several experimental results.
\end{abstract}
 \maketitle

\section{Introduction}

The Kaczmarz method~\cite{Kac37:Angenaeherte-Aufloesung} is an iterative algorithm for solving linear systems of equations.  It is usually applied to large-scale overdetermined systems because of its simplicity and speed (but also converges in the underdetermined case to the least-norm solution under appropriate initial conditions).
Each iteration projects onto the solution space corresponding to one row in the system, in a sequential fashion.  Strohmer and Vershynin prove that when the rows are selected from a certain random distribution rather than sequentially, that the randomized method converges to the solution at a linear rate~\cite{SV09:Randomized-Kaczmarz}.  %
The method has been applied to fields including image reconstruction, digital signal processing, and computer 
tomography~\cite{SS87:Image-Recovery,CFMSS92:POCS-Variants,Nat01:Mathematics-Computerized,RefWorks:504}. Leventhal and Lewis modify the randomized Kaczmarz method 
 to apply to systems of linear 
equalities and inequalities~\cite{LL10:Randomized-Methods}, thereby extending results on the standard method in this setting (see e.g.~\cite{censor1981} and references therein).  Unlike the traditional randomized algorithm which enforces a single constraint at each iteration, the block Kaczmarz 
approach recently analyzed by Needell and Tropp~\cite{NT12:Block-Kaczmarz} enforces multiple constraints simultaneously and thus offers computational advantages. 
Here we demonstrate convergence for a 
system of linear equalities and inequalities by combining a randomized block Kaczmarz method for the equalities with a randomized Kaczmarz 
algorithm for the inequalities. These results indicate that the block Kaczmarz method can be used for a system of equalities and inequalities, 
and in some cases may quicken convergence.  We also consider the case of utilizing blocking in both the equalities and inequalities, although this
can be detrimental unless the geometry of the system meets certain conditions.

\subsection{Model and Notation} \label{sec:assumptions}

We consider a linear system 
\begin{equation} \label{eqn:lin-system}
\mtx{A} \vct{x} = \vct{b},
\end{equation}
where $\mtx{A}$ is a real (or complex) $n \times d$ matrix, typically with $n\gg d$. 

The $\ell_p$ vector norm for $p \in [1, \infty]$ is denoted $\norm{\cdot}_p$, while $\norm{\cdot}$ is the spectral norm and $\fnorm{\cdot}$ refers to the Frobenius norm. 
For an $n \times d$ matrix $\mtx{A}$, the singular values are arranged in decreasing order and we write
$$
\sigma_{\max}(\mtx{A}) \defby \sigma_1(\mtx{A}) \geq \sigma_2(\mtx{A})
	\geq \dots \geq \sigma_{d}(\mtx{A}) \defby \sigma_{\min}(\mtx{A}).
$$
We define the eigenvalues $\lambda_{\min}(\mtx{A}), \ldots, \lambda_{\max}(\mtx{A})$ of a matrix analogously.
For convenience we will assume that each row $\vct{a_i} $ of $\mtx{A}$ has unit norm, 
$\enorm{ \vct{a_i} } = 1$, and we call such matrices \textit{standardized}.

We define the usual condition number $$\kappa(\mtx{A}) \defby \sigma_{\max}(\mtx{A})/\sigma_{\min}(\mtx{A}),$$ and write
the Moore-Penrose pseudoinverse of matrix $\mtx{A}$ by $\mtx{A}^\dagger$. 
Recall that for a matrix $\mtx{A}$ with full row rank, the pseudoinverse is obtained by $\mtx{A}^\dagger \defby \mtx{A}^{*}(\mtx{A}\mtx{A}^{*})^{-1}$.

Now we consider a system of linear equalities and inequalities and denote by $S$ its non-empty set of feasible solutions. 
We thus consider the matrix $\mtx{A}$ whose rows can be arranged such that
\begin{equation} \label{eqn:ineq-matrix}
\mtx{A} = \left[
	\begin{array}{lr} 
		\mtx{A}_{=}\\
		\mtx{A}_{\leq} 
	\end{array}
\right],
\end{equation}
and we will write $I_{=}$ and $I_{\leq}$ to denote the row indices of $\mtx{A}_{=}$ and $\mtx{A}_{\leq}$, respectively.
Therefore, we ask that
\begin{equation} \label{eqn:ineq-system}
\langle \vct{a}_{i}, \vct{x} \rangle \leq \vct{b}_{i} \quad(i \in I_{\leq})\quad\text{and}\quad
\langle \vct{a}_{i}, \vct{x} \rangle = \vct{b}_{i} \quad(i \in I_{=})
\end{equation}
We will assume that the set of rows $\{1,2,...,n\}$ is partitioned such that the first $n_{e}$ rows correspond to equalities, and the remaining $n_i = n-n_{e}$ rows 
to inequalities.  Thus $\mtx{A}_=$ is an $n_e\times d$ matrix and $\mtx{A}_{\leq}$ is $n_i \times d$.

The error bound for this system of linear inequalities uses the function $e: \textbf{R}^n \rightarrow \textbf{R}^n$ defined as in~\cite{LL10:Randomized-Methods} by
 
$$
  e(y)_{i} = \left\{
     \begin{array}{lr}
       y_{i}^{+} & \text{for } i \in {I_{\leq}}\\
       y_{i} & \text{for } i \in {I_{=}}
     \end{array}
   \right.
$$
where the positive part is defined as ${x^{+}} \defby \max(x,0)$.

\subsection{Details of Kaczmarz}
The simple Kaczmarz method is an iterative algorithm that approximates a least-squares minimizer $\vct{x}_{\star}$ to the problem in~\eqref{eqn:lin-system}. 
It takes an arbitrary initial approximation $\vct{x}_{0}$, and at each iteration $j$ the current iterate is projected orthogonally onto the solution 
hyperplane $\{\left\langle \vct{a}_{i}, \vct{x} \right\rangle = \vct{b}_{i}\}$, using the update rule
\begin{equation} \label{eqn:simp-rule}
\vct{x}_{j+1} = \vct{x}_{j} + \frac{\vct{b}_{i} - \left\langle \vct{a}_{i}, \vct{x}_{j} \right\rangle }{\ \enormsq{\vct{a}_{i}}} \vct{a}_{i}
\end{equation}
where $i = j$ mod $n+1$ \cite{Kac37:Angenaeherte-Aufloesung}.  With an unfortunate ordering of the rows, this method as-is can produce very slow convergence.
However, it has been well known that using randomized selection often eliminates this effect \cite{HS78:Angles-Null,HM93:Algebraic-Reconstruction}. 
The randomized Kaczmarz method put forth by Strohmer and Vershynin \cite{SV09:Randomized-Kaczmarz} uses a random selection method for the selection of row $i$ 
such that each row $i$ is selected with probability proportional to $\enormsq{\vct{a}_{i}}$.  This randomization provides an algorithm that is both 
simple to analyze and enforce in many cases.   In this paper we assume each row has unit norm, 
so each row is selected uniformly at random from $\{$1,...,n$\}$ in the simple randomized Kaczmarz approach.\footnote{This assumption is both for notational convenience, and because the use of matrix pavings discussed below only hold for standardized matrices.  In practice, one can employ pre-conditioning on non-standardized systems, or extend the construction of matrix pavings to non-standardized systems~\cite{RefWorks:Versh}.}  Strohmer and Vershynin prove a linear rate of 
convergence for consistent systems that depends on the scaled condition number of $\mtx{A}$, and not on the number of equations $n$ \cite{SV09:Randomized-Kaczmarz}, 
\begin{equation}\label{rv}
\Expect \enormsq{\vct{x}_{j} -\vct{x}_{\star}} \leq \left[1 - \frac{1}{K}\right]^j \enormsq{\vct{x}_{0} - \vct{x}},
\end{equation}
where $\vct{x}_{\star}$ is the solution to the consistent system~\eqref{eqn:lin-system} and $K = \|\mtx{A}\|_F^2 / \sigma^2_{\min}(\mtx{A})$ denotes the scaled condition number.  
Needell extended this work to the inconsistent case and proves linear convergence to the least-squares solution within some fixed radius \cite{Nee10:Randomized-Kaczmarz},
$$
\Expect \enormsq{\vct{x}_{j} -\vct{x}_{\star}} \leq \left[1 - \frac{1}{K}\right]^j \enormsq{\vct{x}_{0} - \vct{x}} + K\|\vct{e}\|_{\infty}^2,
$$
where $\vct{e} = \mtx{Ax_{\star}} - \vct{b}$ denotes the residual vector.  Because the Kaczmarz method projects directly onto each solution hyperplane,
such a convergence radius is unavoidable without adding a relaxation parameter.

The randomized Kaczmarz method can be adapted to the case of a linear system of equalities and inequalities described in~\eqref{eqn:ineq-system}. 
Leventhal and Lewis \cite{LL10:Randomized-Methods} apply the Kaczmarz method to a consistent system of linear equalities and inequalities (here
consistent simply means the feasible set $S$ is non-empty). At each iteration $j$, the previous iterate only projects onto the solution hyperplane 
if the inequality is not already satisfied. If the inequality is satisfied  for row $i$ selected at iteration $j$ $(\vct{a}_{i}^T \vct{x} \leq \vct{b}_{i})$, 
the approximation $\vct{x}_{j}$ is set as $\vct{x}_{j-1}$ \cite{LL10:Randomized-Methods}. The update rule for this algorithm is thus
\begin{equation} \label{eqn:ineq-rule}
\vct{x}_{j+1} = \vct{x}_{j} - \frac{e(\vct{a}_{i}^{T}\vct{x}_{j} - \vct{b}_{i})}{\ \enormsq{\vct{a}_{i}}} \vct{a}_{i}.
\end{equation}

This algorithm converges linearly in expectation \cite{LL10:Randomized-Methods}, with
$$
\Expect \left[ d(\vct{x}_{j},S)^2 \;|\; \vct{x}_{j-1} \right] \leq d(\vct{x}_{j-1},S)^2 - \frac{\enormsq{e(\mtx{A}\vct{x}_{j-1}-\vct{b})}}{\ \normsq{\mtx{A}}_{F}}.
$$

In order to bound the right hand side of this expression, the authors rely on a lemma due to Hoffman \cite{Hof52:Approximate-Solutions,LL10:Randomized-Methods}.  This result 
states that for any system~\eqref{eqn:ineq-system} with non-empty solution set $S$, there exists a constant $L$ independent of $\vct{b}$ such that for all $\vct{x}$,
\begin{equation}\label{hoffman}
d(\vct{x}, S) \leq L\enorm{e(\vct{Ax}-\vct{b})}.
\end{equation}
When $\mtx{A}_= = \mtx{A}$ is full column rank, the Hoffman constant is the inverse of the smallest singular value, $L = \sigma_{\min}^{-1}(\mtx{A})$.

Using this their result becomes
\begin{equation} \label{eqn:ineq-convergence}
\Expect \left[d(\vct{x}_j,S)^2\right]
	\ \leq \ \left[ 1 - \frac{1}{L^2\|\mtx{A}\|_F^2} \right]^j\cdot
	d(\vct{x}_{0},S)^2,
\end{equation}
which coincides with~\eqref{rv} for consistent systems of equalities.

\subsection{Block Kaczmarz}
A block variant of the randomized Kaczmarz method due to Elfving \cite{Elf80:Block-Iterative-Methods} has been recently analyzed by Needell and Tropp \cite{NT12:Block-Kaczmarz} and can improve the convergence rate in 
certain cases. The block Kaczmarz 
method first partitions the rows $\{1,...,n\}$ into $m$ blocks, denoted $\tau_{1}, \ldots \tau_m$. Instead of selecting one row per iteration as done with the simple 
Kaczmarz method, the block Kaczmarz algorithm chooses a block uniformly at random at each iteration. Thus the block Kaczmarz method enforces multiple constraints 
simultaneously. At each iteration, the previous iterate $\vct{x}_{j-1}$ is projected onto the solution space to $\mtx{A}_{\tau}\vct{x} = \vct{b}_{\tau}$, which enforces 
the set of equations in block $\tau$ \cite{NT12:Block-Kaczmarz}. $\mtx{A}_{\tau}$ and $\vct{b}_{\tau}$ are written as the row submatrix of $\mtx{A}$ and the subvector 
of $\vct{b}$ indexed by $\tau$ respectively, yielding an iterative rule of
\begin{equation} \label{eqn:block-algorithm}
\vct{x}_{j} = \vct{x}_{j-1} + (\mtx{A}_{\tau})^\dagger(\vct{b}_{\tau} - \mtx{A}_{\tau} \vct{x}_{j-1}).
\end{equation}
The pseudoinverse used in~\eqref{eqn:block-algorithm} returns the solution to the underdetermined least squares problem for a wide or square row submatrix $\mtx{A}_{\tau}$.%

Depending on the characteristics of the submatrix $\mtx{A}_{\tau}$, the block method can provide better convergence than the 
simple method.  
If we assume that the submatrices $\mtx{A}_{\tau}$ are well conditioned, the additional cost of computing their pseudo-inverse can be overcome by the gain in utilizing block multiplications (see our experiments in Section~\ref{sec:exps}).  In fact, if the blocks admit a fast multiply (for example if the matrix is built of DFT or circulant blocks), then 
the computational cost of the block iteration~\eqref{eqn:block-algorithm} is similar to the cost of the simple update 
rule in~\eqref{eqn:simp-rule}. 
Since the convergence depends heavily on the conditioning of each submatrix, one seeks partitions of the rows into
blocks for which each block is well-conditioned.  The notion of a \textit{row-paving} allows one to do precisely that.

\begin{definition}
We define an ($m$, $\beta$) row paving\footnote{The standard definition of a row paving also includes a constant $\alpha$ which serves as a lower
bound to the smallest singular value.  We ignore that parameter here since it will not be utilized.} of matrix $\mtx{A}$ as a partition 
$ T=\{\tau_1,...\tau_m\}$ of the row indices such that 
$$
\lambda_{\max}( \mtx{A}_{\tau} \mtx{A}_{\tau}^* ) \leq \beta
\quad\text{for each $\tau \in T$.}
$$
\end{definition}
The \textit{size} of the paving, or number of blocks, is $m$.  The value of $\beta$ is the upper paving bound, which 
controls the spectral norms of the submatrices. 
Needell and Tropp \cite{NT12:Block-Kaczmarz} show that these parameters determine the performance of the algorithm, 
with convergence for a consistent system admitting an $(m, \beta)$ paving given by
\begin{equation} \label{eqn:block-convergence}
\Expect \enormsq{ \vct{x}_j - \vct{x}_{\star} }
	\ \leq \ \left[ 1 - \frac{\sigma_{\min}^2(\mtx{A})}{\beta m} \right]^{j}
	\enormsq{ \vct{x}_0 - \vct{x}_{\star} }.
\end{equation}
Therefore the convergence rate depends on the size $m$ and upper bound $\beta$; the algorithm's performance improves with low values of $m$ and $\beta$, 
and large $\sigma_{\min}^2(\mtx{A})$.  The authors also prove convergence for inconsistent systems, with the same convergence rate and convergence radius
which depends also on the minimum of all $\lambda_{\min}( \mtx{A}_{\tau} \mtx{A}_{\tau}^* )$, see \cite{NT12:Block-Kaczmarz} for details.

Surprisingly, every standardized matrix admits a good row paving.  The following result is due to \cite{RefWorks:485,Tro09:Column-Subset} which builds off the
foundational work of \cite{BT87:Invertibility-Large,RefWorks:546}.
\begin{proposition}[Existence of Good Row Pavings] \label{prop:intro-paving}
For any $\delta \in (0, 1)$ and standardized $n \times d$ matrix $\mtx{A}$, there is a row paving satisfying
$$
m \leq \cnst{C} \cdot \delta^{-2} \normsq{\mtx{A}} \log(1+n)
\quad\text{and}\quad
1 - \delta \leq \beta \leq 1 + \delta.
$$
where $\cnst{C}$ is an absolute constant.%
\end{proposition}

Although this is an existential result, there are constructive methods to obtain such pavings, and for certain classes of matrices, 
they can even be obtained by a random partitioning of the rows \cite{RefWorks:556,RefWorks:536,NT12:Block-Kaczmarz}.

With such a paving in tow, the convergence of~\eqref{eqn:block-convergence} becomes
$$
\Expect \normsq{ \vct{x}_j - \vct{x}_{\star} }
	\leq \left[ 1 - \frac{1}{\cnst{C} \kappa^2(\mtx{A}) \log(1+n)} \right]^j
	\enormsq{ \vct{x}_0 - \vct{x}_\star }
	\
$$

Although often comparable to the convergence rate for the simple method~\eqref{rv}, numerical results confirm
that the block method offers significant reduction in computation time due to the speed of matrix--vector multiplication (see e.g. \cite{NT12:Block-Kaczmarz}).

\subsection{Contribution}

This paper analyzes the system with matrix described in~\eqref{eqn:ineq-matrix} using an algorithm with the block Kaczmarz 
approach for the equalities given by $\mtx{A}_{=}$ and the simple method for the inequalities given by $\mtx{A}_{\leq}$. 
A paving is created for $\mtx{A}_{=}$, with the inequalities excluded. At each iteration, we select from $\mtx{A}_{=}$ with a fixed probability
$p$ and from $\mtx{A}_\leq$ with probability $1-p$.  In the former case, we select a block $\tau$ from paving $T$ uniformly at random, 
and in the latter case we select a row $i$ of $\mtx{A}_\leq$ uniformly at random. 
In the case of a block of equalities being selected, the algorithm proceeds by updating $\vct{x}_{j}$ using~\eqref{eqn:block-algorithm}. 
When an inequality row is selected, $\vct{x}_{j}$ is updated using the rule~\eqref{eqn:ineq-rule}.  We prove that this method yields
linear convergence to the solution set $S$.  We also include a discussion about paving both $\mtx{A}_=$ and $\mtx{A}_\leq$,
which identifies a geometric property of the system which allows for
improved convergence by utilizing two pavings.  We show that when this property is not satisfied, utilizing both pavings can be detrimental to convergence.

\subsection{Organization}
Section~\ref{sec:main} lays out our main result, Theorem~\ref{thm:convergence}, and provides a proof. 
We discuss blocking the full matrix in Section~\ref{sec:double} and
Section~\ref{sec:exps} explains numerical experiments and results.  We conclude with discussion and related work in Section~\ref{sec:end}.

\section{Analysis of the Block Kaczmarz Algorithm for a System of Inequalities}\label{sec:main}

In this section we analyze the convergence of the described method, which is detailed in Algorithm~\ref{alg}.

\begin{center}
\begin{algorithm}[hb]
\caption{Block Kaczmarz Method for a System of Inequalities}
	\label{alg}
\begin{center} \fbox{
\begin{minipage}{.95\textwidth} 
\vspace{4pt}
\alginout{\begin{itemize}
\item	Matrix $\mtx{A}$ with dimension $n \times d$
\item	Right-hand side $\vct{b}$ with dimension $n$
\item Number of rows representing equalities, $n_{e}$, and inequalities, $n_i = n - n_e$
\item	Partition $T = \{\tau_1, \dots, \tau_m\}$ of the row indices $\{1, \dots, n_{e}\}$ and paving constant $\beta$
\item	Initial iterate $\vct{x}_0$ with dimension $d$
\item	Convergence tolerance $\varepsilon > 0$
\end{itemize}}
{An estimate $\hat{\vct{x}}$ to the solution of the system~\eqref{eqn:ineq-system}}
\vspace{8pt}\hrule\vspace{8pt}

\begin{algtab*}
$j \leftarrow 0$

\algrepeat 
	$j \leftarrow j + 1$ \\
	Draw uniformly at random $q$ from $[0,1]$ \\
	\algif{$q \leq \frac{\beta m}{n_{i} + \beta m}$}
		Choose a block $\tau$ uniformly at random from $T$ \\
		$\vct{x}_j \leftarrow \vct{x}_{j-1} + (\mtx{A}_\tau)^\pinv (\vct{b}_\tau - \mtx{A}_\tau\vct{x}_{j-1})$ \\%\hfill \hfill  ({Solve least-squares approximation}) \\ %\\
	\algelse
		Choose a row $i$ uniformly at random from $\{n_{e} + 1, \dots, n\}$ \\
		$\vct{x}_{j} \leftarrow \vct{x}_{j-1} - \frac{e(\vct{a}_{i}^{T}\vct{x}_{j-1} - \vct{b}_{i})}{\ \enormsq{\vct{a}_{i}}} \vct{a}_{i}$ \\
	\algend
\alguntil{$\enormsq{ e(\mtx{A}\vct{x}_j - \vct{b}) } \leq \eps^2$}
$\hat{\vct{x}} \leftarrow \vct{x}_j$
\end{algtab*}
\end{minipage}}
\end{center}

\end{algorithm}
\end{center}

Notice that the probability of selecting a block of $\mtx{A}_=$ is $\frac{\beta m}{n_{i} + \beta m}$.  This quantity corresponds to the relative
size of $A_{=}$ in the system, where the size is measured in terms of the paving quantities $\beta m$.  This value may be difficult to compute
precisely, and the simpler threshold of $n_e/n$ appears to also work well in practice.  
We provide no evidence that our selection of this threshold is most efficient, nor any more efficient than using one proportional 
to the number of equality rows $n_{e}$.
We find that this algorithm yields linear convergence in expectation with a rate that only depends on the number of inequalities $n_{i}$, paving size $m$, 
and upper bound $\beta$.

Our main result is described in Theorem~\ref{thm:convergence}.  

\begin{theorem}[Convergence] \label{thm:convergence}
Let the standardized matrix $\mtx{A}\in\mathbb{R}^{n \times d}$ and $b\in\mathbb{R}^n$ correspond to a system as in~\eqref{eqn:ineq-matrix} with the first $n_{e}$ rows being equalities and the 
remaining $n_i = n - n_{e}$ rows being inequalities.  Let $T$ be an $(m, \beta)$ row paving of $\mtx{A}_{=}$.   Let $\vct{x}_0$ be an arbitrary
 initial estimate and $S$ the non-empty feasible region. Then Algorithm~\ref{alg} 
satisfies for each iteration  $j$ = 1,2,3,...,
\begin{align*}
\Expect \left[d(\vct{x}_j,S)^2\right]
	\ \leq \ \left( 1 - \frac{1}{L^2(n_{i} + \beta m)} \right)^j\cdot
	d(\vct{x}_{0},S)^2,
\end{align*}
where $L$ is the Hoffman constant~\eqref{hoffman}.
\end{theorem}

\begin{remarks}\\
\noindent{\bfseries 1. }Note that when there are no block projections, no inequalities, or neither, Theorem~\ref{thm:convergence} recovers
the results of the standard randomized Kaczmarz for inequalities \cite{LL10:Randomized-Methods}, the standard randomized block Kaczmarz method \cite{NT12:Block-Kaczmarz} or the standard 
randomized Kaczmarz method \cite{SV09:Randomized-Kaczmarz}, respectively.  We thus view this result as a completely generalized convergence bound.

\noindent{\bfseries 2. } If we let $\rho_s$ and $\rho_b$ be the convergence rates of the simple and block methods for mixed systems, respectively, then 
by~\eqref{eqn:ineq-convergence} and Theorem~\ref{thm:convergence},
$$
\rho_{{s}} \geq \frac{1}{L^2 n} \quad\text{ and }\quad \rho_{{b}} \geq \frac{1}{L^2( n_{i} + \beta m)}.
$$ 
It is evident that our expected convergence rate will be faster per iteration than the simple method when $n_{i} + \beta m < n$.  Since $\beta$ can be 
chosen close to $1$ and $m<n_e$ is then number of rows in $\mtx{A}_=$, this holds quite easily.

\noindent{\bfseries 3. }Since a single iteration using a block $\mtx{A}_\tau$ in general may cost more than an iteration utilizing a single row,
it is more fair to compare per epoch, rather than per iteration.  An epoch is typically the minimum number of iterations needed to visit each row of the
matrix.  When there are inequalities present that are already satisfied in a given iteration, that iteration may make no contribution and cost
very little computationally.  Thus the notion of epoch may be slightly skewed here, but if we ignore this subtlety
the simple method will have approximately $n$ 
iterations per epoch, compared to $n_{i} + m$ iterations per epoch with the block method. The approximate per epoch convergence rates can thus be compared as
$$
n \cdot \rho_{{s}} \geq \frac{1}{L^2} \quad\text{ and }\quad (n_{i} + m) \cdot \rho_{{b}} \geq \frac{n_{i} + m}{L^2( n_{i} + \beta m)}.
$$ 
This result is similar to that found by Needell and Tropp \cite{NT12:Block-Kaczmarz}, with the block convergence rate at best equal 
to that of the simple convergence rate when $\beta = 1$.  However, as already noted, the block method is quite advantageous computationally. 
\end{remarks}

Combining the paving result of Prop.~\ref{prop:intro-paving} with Theorem~\ref{thm:convergence} yields the following corollary.

\begin{corollary}\label{cor1}
Instate the assumptions and notation of Theorem~\ref{thm:convergence} and let $\mtx{A}_=$ be equipped with an $(m, \beta)$ row-paving as in Proposition~\ref{prop:intro-paving}.
Then the iterates of Algorithm~\ref{alg} satisfy
\begin{align*}
\Expect \left[d(\vct{x}_j,S)^2\right]
	\ \leq \ \gamma^j\cdot
	d(\vct{x}_{0},S)^2,
\end{align*}
where $\gamma = \left( 1 - \frac{1}{L^2(n_{i} + C\|\mtx{A}_=\|^2\log(1+n))} \right)$ and $C$ is some absolute constant.
\end{corollary}

\begin{proof}[of Theorem~\ref{thm:convergence}]
Fix an iteration $j$ of Algorithm~\ref{alg}.  We proceed as in \cite{NT12:Block-Kaczmarz} and \cite{LL10:Randomized-Methods}.
First, we suppose that $q \leq \frac{\beta m}{n_{i} + \beta m}$, so that a block $\tau$ of equalities is selected this iteration.  Then writing $P_S$ as the orthogonal projection onto $S$, we have
 $\mtx{b}_\tau = \mtx{A}_\tau P_S \vct{x}_{j-1}$ since $P_S \vct{x}_{j-1} \in S$.   We then have 
\begin{align*}
\vct{x}_j &= \vct{x}_{j-1} + \mtx{A}_\tau^\pinv (\mtx{b}_\tau - \mtx{A}_\tau \vct{x}_{j-1}) \\
&= \vct{x}_{j-1} + \mtx{A}_\tau^\pinv (\mtx{A}_\tau P_S \vct{x}_{j-1} - \mtx{A}_\tau \vct{x}_{j-1}) \\
&= \vct{x}_{j-1} + \mtx{A}_\tau^\pinv \mtx{A}_\tau (P_S \vct{x}_{j-1} - \vct{x}_{j-1}).
\end{align*}

Thus,
\begin{align*}
\|\vct{x}_j &- P_S \vct{x}_{j-1}\|^2\\
 &= \enormsq{\vct{x}_{j-1} - P_S \vct{x}_{j-1} - \mtx{A}_\tau^\pinv \mtx{A}_\tau (\vct{x}_{j-1} - P_S \vct{x}_{j-1})}  \\
&= \enormsq{(\Id - \mtx{A}_\tau^\pinv \mtx{A}_\tau) (\vct{x}_{j-1} - P_S \vct{x}_{j-1})}.
\end{align*}

Taking expectation (over the choice of the block $\tau$, conditioned on previous choices), and using the fact that $\mtx{A}_\tau^\pinv \mtx{A}_\tau$ is an orthogonal projector, along with the properties of the paving yields

\begin{align*}
\Expect \|\vct{x}_j &- P_S \vct{x}_{j-1}\|^2\\
&= \Expect \enormsq{(I - \mtx{A}_\tau^\pinv \mtx{A}_\tau) (\vct{x}_{j-1} - P_S \vct{x}_{j-1})}  \\
&= \enormsq{\vct{x}_{j-1} - P_S \vct{x}_{j-1}} - \Expect \enormsq{\mtx{A}_\tau^\pinv \mtx{A}_\tau (\vct{x}_{j-1} - P_S \vct{x}_{j-1})}  \\
&\leq \enormsq{\vct{x}_{j-1} - P_S \vct{x}_{j-1}} - \frac{1}{\ \beta} \Expect \enormsq{\mtx{A_{\tau}} (\vct{x}_{j-1} - P_S \vct{x}_{j-1})}. 
\end{align*}

Since $d(\vct{x}_{j-1},S) = \enorm{\vct{x}_{j-1} - P_S \vct{x}_{j-1}}$ and $d(\vct{x}_{j}, S) \leq \enorm{\vct{x}_j - P_S \vct{x}_{j-1}}$, this means that

\begin{align}\label{eqs}
\Expect \left[d(\vct{x}_j,S)^2\right] &\leq 
d(\vct{x}_{j-1}, S)^2 - \frac{1}{\ \beta } \Expect \enormsq{\mtx{A_{\tau}} (\vct{x}_{j-1} - P_S \vct{x}_{j-1})} \notag\\
&= d(\vct{x}_{j-1}, S)^2 - \frac{1}{\ \beta m} \sum_{\tau \in T} \enormsq{\mtx{A}_{\tau} \vct{x}_{j-1} - \vct{b}_{\tau}} \notag\\
&= d(\vct{x}_{j-1}, S)^2 - \frac{1}{\ \beta m} \sum_{i \in I_{=}} e(\mtx{A}_{=} \vct{x}_{j-1} - \vct{b}_{=})_{i}^2.
\end{align}

Next suppose that instead $i \in I_\leq$ is selected.  Then since each row $\vct{a_i}$ has unit norm,

\begin{align*}
d(\vct{x}_j,S)^2 &\leq \enormsq{\vct{x}_j - P_S \vct{x}_{j-1}}\\
&= \enormsq{\vct{x}_{j-1} - e(\mtx{Ax}_{j-1}-\vct{b})_i \vct{a_i} - P_S \vct{x}_{j-1}}\\
&= \enormsq{\vct{x}_{j-1} - P_S \vct{x}_{j-1}} + e(\mtx{Ax}_{j-1}-\vct{b})_i^2  \\
&\;\;-2e(\mtx{Ax}_{j-1}-\vct{b})_i\langle \vct{a_i} , \vct{x}_{j-1} - P_S \vct{x}_{j-1} \rangle\\
&\leq d(\vct{x}_{j-1},S)^2 - {e(\mtx{A} \vct{x}_{j-1} - \vct{b})_{i}^2},
\end{align*}

where the last line follows from the fact that $\langle \vct{a_i}, P_S \vct{x}_{j-1} \rangle \leq b_i$ and $e(\mtx{A} \vct{x}_{j-1} - \vct{b})_{i} \geq 0$.
Now taking expectation again we have

\begin{align*}
\Expect \left[d(\vct{x}_j,S)^2\right] &\leq d(\vct{x}_{j-1},S)^2 - \Expect (e(\mtx{A} \vct{x}_{j-1} - \vct{b})_{i}^2) \\
&= d(\vct{x}_{j-1},S)^2 - \frac{1}{\ n_{i}} \sum_{i \in I_\leq} e(\mtx{A}_{\leq} \vct{x}_{j-1} - \vct{b}_\leq)_{i}^2. 
\end{align*}

Combining these results and letting $E_=$ and $E_\leq$ denote the events that a block from $T$ and a row from $I_{\leq}$ is selected, respectively, we have

\begin{align*}
\Expect \left[d(\vct{x}_j,S)^2\right] &= p \cdot \Expect[ d(\vct{x}_{j},S)^2 | E_=] + (1-p) \cdot \Expect [ d(\vct{x}_{j},S)^2 | E_{\leq}] \\
&\leq p\left[d(\vct{x}_{j-1}, S)^2 - \frac{1}{\ \beta m} \sum_{i \in I_=} e(\mtx{A}_{=} \vct{x}_{j-1} - \vct{b}_=)_{i}^2\right] \\
&\;\;+ (1-p) \left[d(\vct{x}_{j-1},S)^2 - \frac{1}{\ n_{i}} \sum_{i \in I_\leq} e(\mtx{A}_{\leq} \vct{x}_{j-1} - \vct{b}_\leq)_{i}^2\right] \\
&= d(\vct{x}_{j-1}, S)^2 - p \cdot \frac{1}{\ \beta m} \sum_{i \in I_=} e(\mtx{A}_{=} \vct{x}_{j-1} - \vct{b}_=)_{i}^2\\
&\;\; - (1-p) \cdot \frac{1}{\ n_{i}} \sum_{i \in I_\leq} e(\mtx{A}_{\leq} \vct{x}_{j-1} - \vct{b}_\leq)_{i}^2.
\end{align*}

Since $p = \frac{\beta m}{\ n_{i} + \beta m}$, we have $\frac{1-p}{n_i} = \frac{1}{n_{i}+\beta m}$ and we can simplify

\begin{align*}
\Expect \left[d(\vct{x}_j,S)^2\right] &\leq d(\vct{x}_{j-1}, S)^2 - \frac{1}{\ n_{i} + \beta m} \Big[\sum_{i \in I_=} e(\mtx{A}_{=} \vct{x}_{j-1} - \vct{b}_=)_{i}^2 \\
&\;\;\;+ \sum_{i \in I_\leq} e(\mtx{A}_{\leq} \vct{x}_{j-1} - \vct{b}_\leq)_{i}^2\Big]  \\
&= d(\vct{x}_{j-1}, S)^2 - \frac{1}{\ n_{i} + \beta m} \enormsq{e(\mtx{A} \vct{x}_{j-1} - \vct{b})} \\
&\leq d(\vct{x}_{j-1}, S)^2 - \frac{1}{L^2(n_{i} + \beta m)} \cdot d(\vct{x}_{j-1},S)^2  \\
&= \left[1 - \frac{1}{L^2(n_{i} + \beta m)}\right] d(\vct{x}_{j-1},S)^2,
\end{align*}
 where we have utilized the Hoffman bound~\eqref{hoffman} in the second inequality.
 
 Utilizing independence of the random selections and recursing on this relation yields the desired result.
\end{proof}

\section{A Discussion about Blocking Inequalities}\label{sec:double}
It is natural to ask whether one can benefit by blocking both the equalities as above and also the inequalities, as described by Algorithm~\ref{alg2}.  Indeed, Section~\ref{sec:exps} will show dramatic improvements 
in computational time when the rows of $\mtx{A}_{=}$ are paved and block projections as in Algorithm~\ref{alg} are used.  So can one benefit even more by paving also the rows of $\mtx{A}_{\leq}$?  
The answer to this question heavily depends on the structure of the matrix $\mtx{A}$. 

\begin{center}
\begin{algorithm}[hb]
\caption{Double Block Kaczmarz Method for a System of Inequalities}
	\label{alg2}
\begin{center} \fbox{
\begin{minipage}{.95\textwidth} 
\vspace{4pt}
\alginout{\begin{itemize}
\item	Matrix $\mtx{A}$ with dimension $n \times d$
\item	Right-hand side $\vct{b}$ with dimension $n$
\item	Partition $T' = \{\tau_1', \dots, \tau_{m'}'\}$ of the row indices $\{1, \dots, n_{i}\}$
\item	Partition $T = \{\tau_1, \dots, \tau_m\}$ of the row indices $\{1, \dots, n_{e}\}$
\item	Initial iterate $\vct{x}_0$ with dimension $d$
\item	Convergence tolerance $\varepsilon > 0$
\end{itemize}}
{An estimate $\hat{\vct{x}}$ to the solution of the system~\eqref{eqn:ineq-system}}
\vspace{8pt}\hrule\vspace{8pt}

\begin{algtab*}
$j \leftarrow 0$

\algrepeat 
	$j \leftarrow j + 1$ \\
	Draw uniformly at random $q$ from $[0,1]$ \\
	\algif{$q \leq \frac{\beta m}{\beta' m' + \beta m}$}
		Choose a block $\tau$ uniformly at random from $T$ \\
		$\vct{x}_j \leftarrow \vct{x}_{j-1} + (\mtx{A}_\tau)^\pinv (\vct{b}_\tau - \mtx{A}_\tau\vct{x}_{j-1})$ \hfill  ({Solve least-squares approximation})\\ %\\
	\algelse
		Choose a block $\tau'$ uniformly at random from $T'$ \\
		Set $\sigma = \{i\in\tau' : \langle \vct{a_i}, \vct{x}_{j-1} \rangle > b_i\} \subset\tau'$ \hfill (Select unsatisfied subset)\\
		$\vct{x}_j \leftarrow \vct{x}_{j-1} + (\mtx{A}_\sigma)^\pinv (\vct{b}_\sigma - \mtx{A}_\sigma\vct{x}_{j-1})$ \hfill  ({Solve least-squares approximation})\\
	\algend
\alguntil{$\enormsq{ e(\mtx{A}\vct{x}_j - \vct{b}) } \leq \eps^2$}
$\hat{\vct{x}} \leftarrow \vct{x}_j$
\end{algtab*}
\end{minipage}}
\end{center}

\end{algorithm}
\end{center} 

If we only consider $\mtx{A}_{=}$, a block projection as in~\eqref{eqn:block-algorithm} enforces all the equations indexed by $\tau$ to be satisfied.  This is of course desirable
when the rows indexed by $\tau$ correspond to equalities.  Also, if a single inequality corresponding to row $i$ in $\mtx{A}_{\leq}$ is not satisfied and we perform 
a single projection as in~\eqref{eqn:simp-rule}, we are again enforcing that inequality to hold with equality.  However, this improves the estimation 
since in this case we know the solution
set $S$ lies on the opposite side of the hyperplane $\{\vct{x} : \langle \vct{x}, \vct{a_i} \rangle = b_i\}$ as the current estimation (see Figure~\ref{draw1} (a)).
\begin{figure}[h]
\begin{tabular}{ccc}
\includegraphics[scale=0.4]{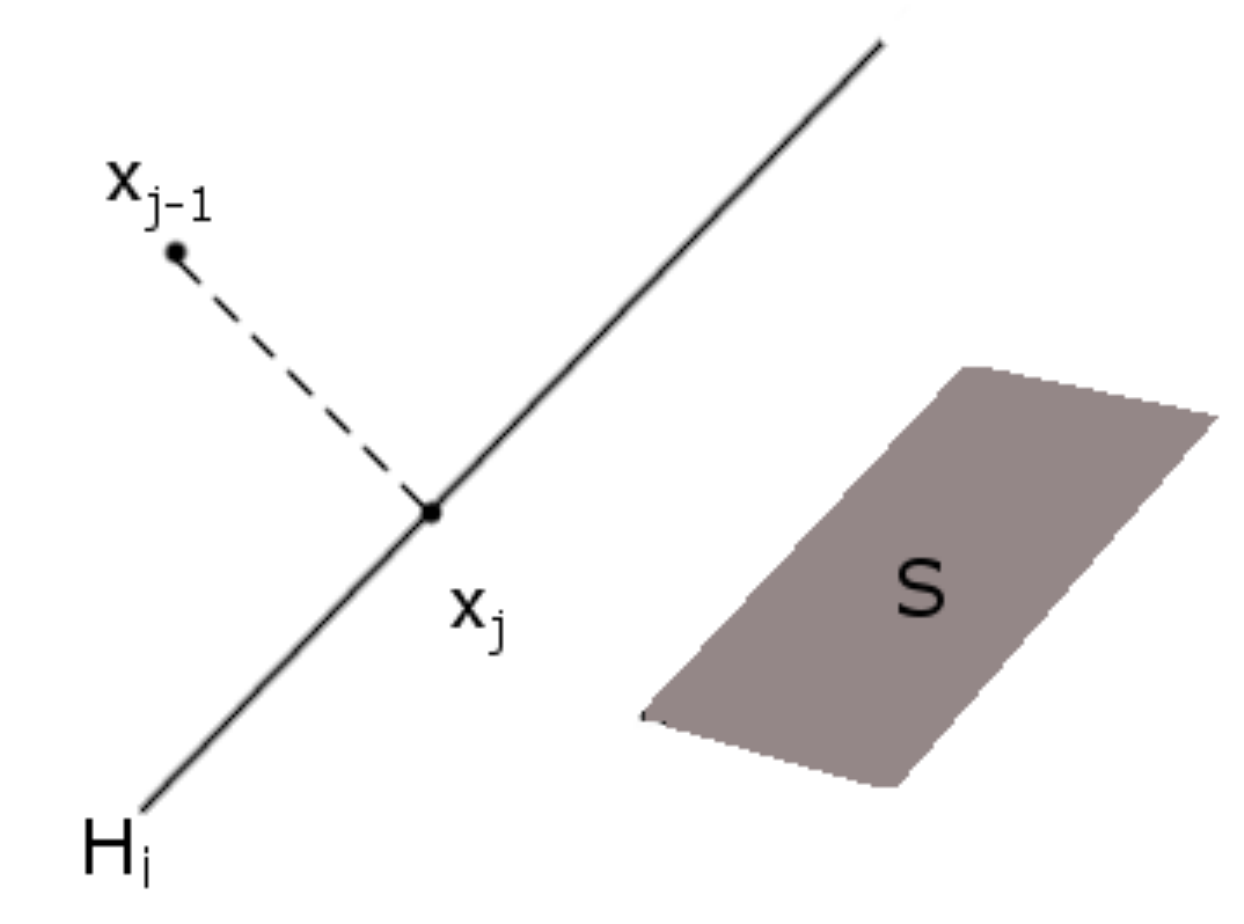}& \includegraphics[scale=0.4]{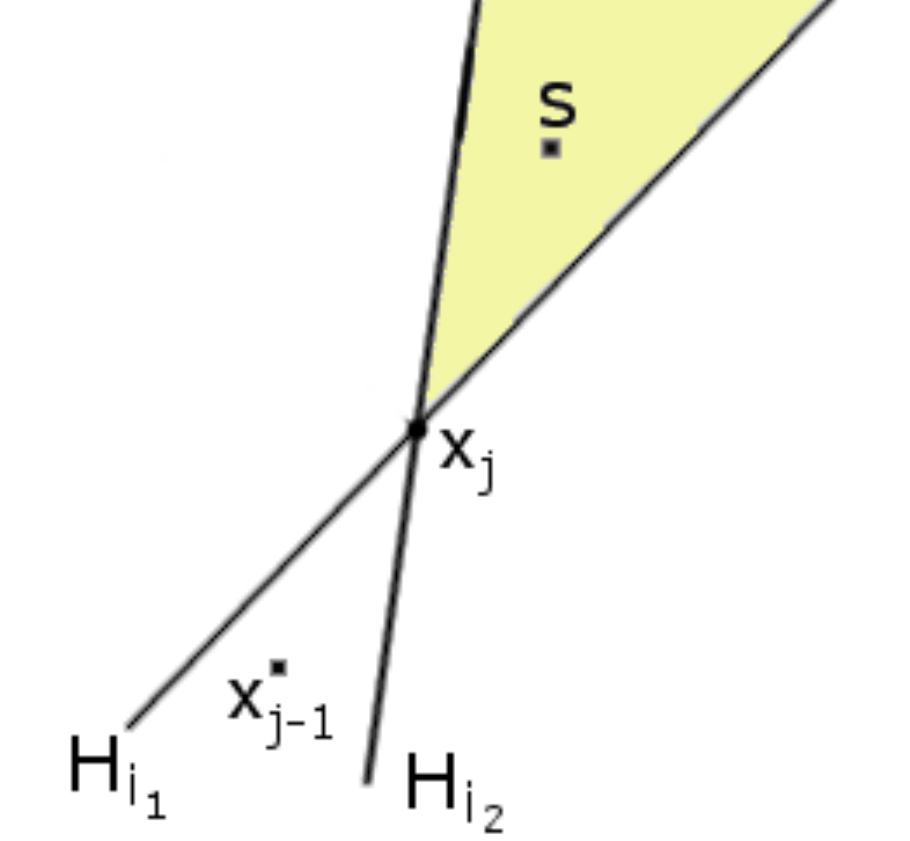}\hspace*{0.5in} & \includegraphics[scale=0.4]{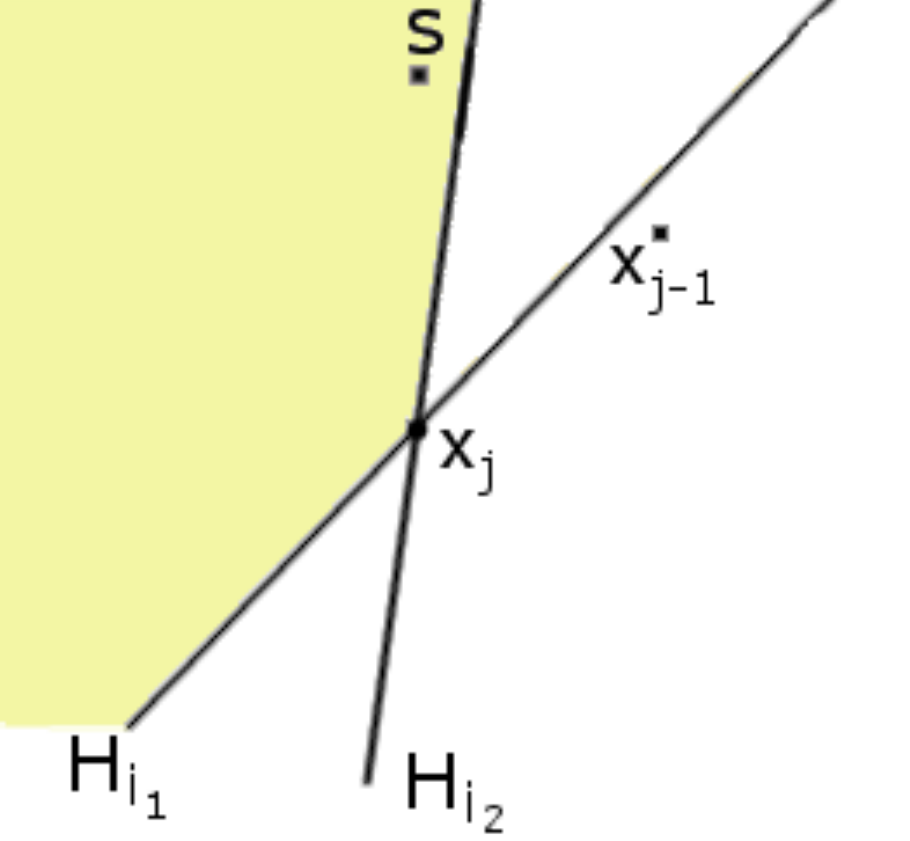}\\
(a) & (b) & (c)\\
\end{tabular}
\caption{Possible geometries of the system. S denotes solution space (or solution point).  Yellow shading denotes regions where inequalities $i_1$ and $i_2$ are
both satisfied. (a) A single projection onto hyperplane $H_i = \{\vct{x} : \langle \vct{a_i}, \vct{x}\rangle = b_i\}$ provides improved estimation.  
(b) Block projection onto intersection of hyperplanes also may provide improved estimation.  (c) Block projection onto intersection of hyperplanes may provide improved estimation.}
\label{draw1}
\end{figure}
On the other hand, if we employ a block projection as in~\eqref{eqn:block-algorithm} to a set of inequalities indexed by $\tau$ which are not satisfied by the current estimation
$\vct{x}_{j-1}$ then we enforce \textit{all of them to hold with equality simultaneously}.  Depending on the geometry of the involved rows, this may result in an improved
estimation or actually one much farther from the solution set.  Of course, one might alternatively want to solve the convex program to project onto
the intersection of the corresponding half-spaces, but we would like to maintain the efficiency and simplicity of the block Kaczmarz method.  

As an illustrative example, Figure~\ref{draw1} (b) and (c) demonstrate two possible scenarios in two dimensions.  Here, the
solution space is a single point marked $S$, and we draw two hyperplanes $H_{i_1}$ and $H_{i_2}$ where $H_i = \{\vct{x} : \langle \vct{a_i}, \vct{x}\rangle = b_i\}$ .
The yellow shaded regions denote areas where both inequalities hold true: $\{\vct{x} : \langle \vct{a}_{i_1}, \vct{x}\rangle \leq b_{i_1} \text{ and } \langle \vct{a}_{i_2}, \vct{x}\rangle \leq b_{i_2}\}$.
 Notice that in (b),
when the angle between $\vct{x}_{j-1} - \vct{x}_j$ and $\vct{s} - \vct{x}_j$ is obtuse, the orthogonal projection of estimation $\vct{x}_{j-1}$ onto 
their intersection is 
guaranteed to be closer to the solution set S.
On the other hand, when that angle is acute we see exactly the opposite, as in (c).  We can quantify this notion by the following definition.

\begin{definition}
For an $r\times d$ matrix $\mtx{A}$ and $\vct{b}\in\mathbb{R}^r$, for row $i$ denote by $\tilde{H}_i$ and $H_i$ the half-space 
$\tilde{H}_i = \{\langle \vct{a_i}, \vct{x}\rangle \leq b_i \}$ and hyperplane ${H}_i = \{\langle \vct{a_i}, \vct{x}\rangle = b_i \}$, 
respectively, and write $P_S$
as the orthogonal projection onto a convex set $S$.
An \textit{obtuse $(m, \beta)$ row  paving} of the matrix $\mtx{A}$ is an $(m, \beta)$ row paving $T=\{\tau_1, \ldots, \tau_m\}$ that also satisfies the following.  
Let $\tau\in T$ and let $\vct{s} \in \cap_{i\in\tau}\tilde{H}_i$, $\vct{w} \in \cap_{i\in\tau}\tilde{H}_i^c$, and 
$\vct{z} = P_{\cap_{i\in\tau}{H}_i}\vct{w}$.  Then 
$$
\langle \vct{w}-\vct{z}, \vct{s}\rangle < 0.
$$
In other words, the angle between $\vct{w}-\vct{z}$ and $\vct{s}$ (and thus $\vct{s} - \vct{z}$) is obtuse.
\end{definition}

We will see that performing block projections on the inequalities in the system only makes sense when one can obtain an obtuse row paving. 
We will use $\vct{w}=\vct{x}_{j-1}$, $\vct{z} = \vct{x}_j$, and $\vct{s}\in S$. 
Notice that if $i_1, i_2 \in \tau \in T$, 
then the partition used in the system depicted in Figure~\ref{draw1} (c) does not constitute an obtuse row paving. 

We conduct two simple experiments to demonstrate the different behavior of the algorithm.  In all cases the matrix $\mtx{A}$ is a $300\times 100$
matrix with standard normal entries, $100$ rows correspond to inequalities, and $\vct{b}$ is generated so that the solution set $S$ is non-empty.  
We measure the residual error which we define as $\enorm{e(\mtx{A}\vct{x}_j - \vct{b})}$.
Figure~\ref{fig:bad} (a) shows the
behavior of the block method with this matrix and a row paving obtained via a random row partition of $30$ blocks ($10$ rows per block).
This generation will create a matrix with paving that with very high probability is not an obtuse row paving.  As Figure~\ref{fig:bad} 
demonstrates, the block method does not converge to a solution in this case.  However, as Figure~\ref{fig:bad} (c) shows, the simple Kaczmarz method succeeds in identifying
a point in the solution space.  Next, we create a matrix in the exact same way, and create the same random row paving.  Then, however,
we iterate through every block in the paving corresponding to inequalities and if two rows $i$ and $k$ in a block satisfy $\langle \vct{a_i}, \vct{a}_k\rangle >0$, we
replace row $\vct{a_i}$ with $-\vct{a_i}$ and entry $b_i$ with $-b_i$.  This guarantees every block in the paving yields a geometry
like that shown in Figure~\ref{draw1} (b), and gives an obtuse row paving. Note that of course this changes the solution space
as well so one cannot employ this strategy in general.  We then add positive values to the entries in $\vct{b}$ 
corresponding to inequalities to ensure the solution set $S$ is non-empty.  With this new system and paving, we again run the block method and 
see that the method
now converges to a point in the solution set, as seen in Figure~\ref{fig:bad} (b).
 
 \begin{figure}[ht]
\begin{center}
\begin{tabular}{ccc}
\includegraphics[scale=0.35]{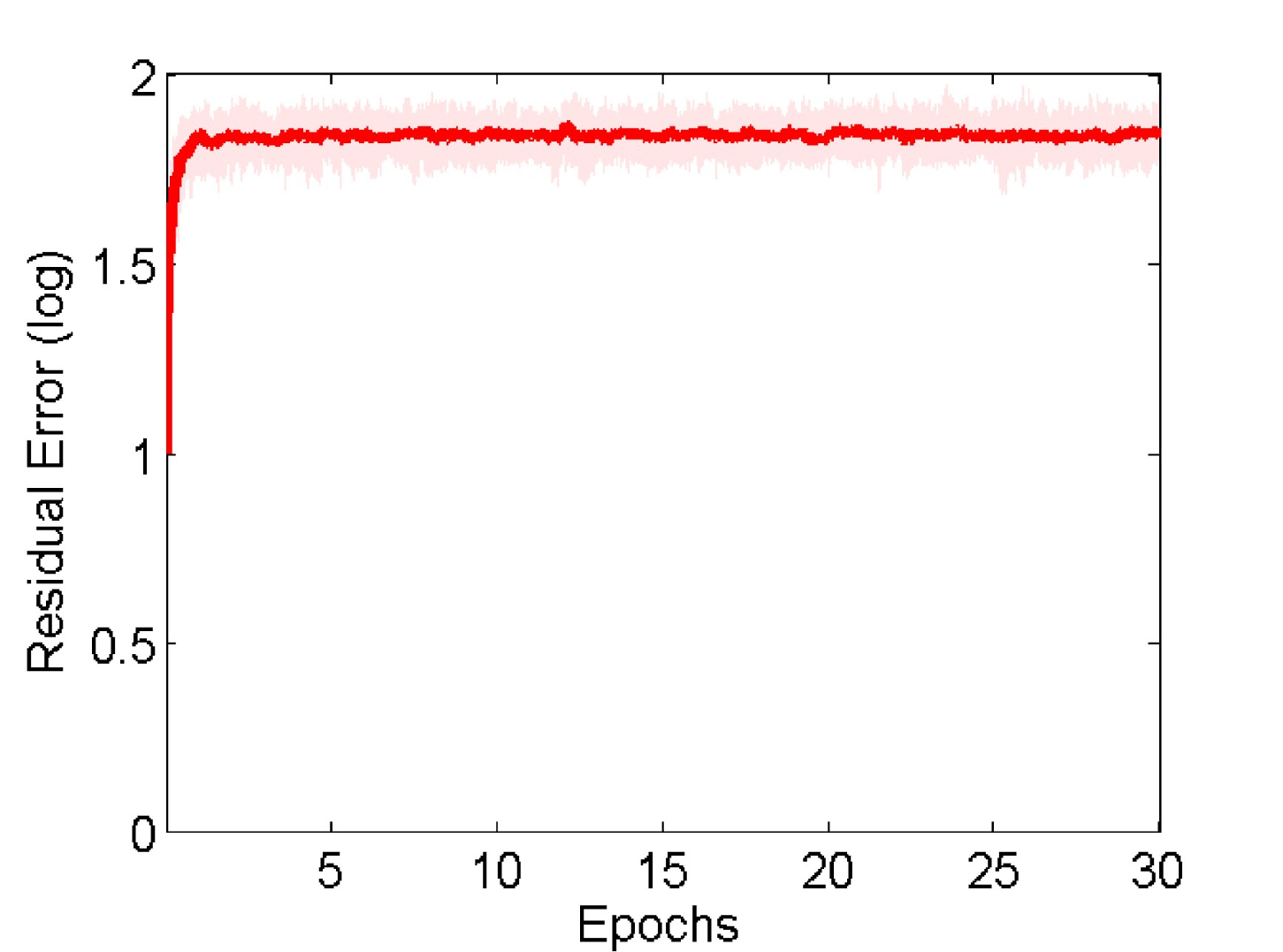} & \includegraphics[scale=0.35]{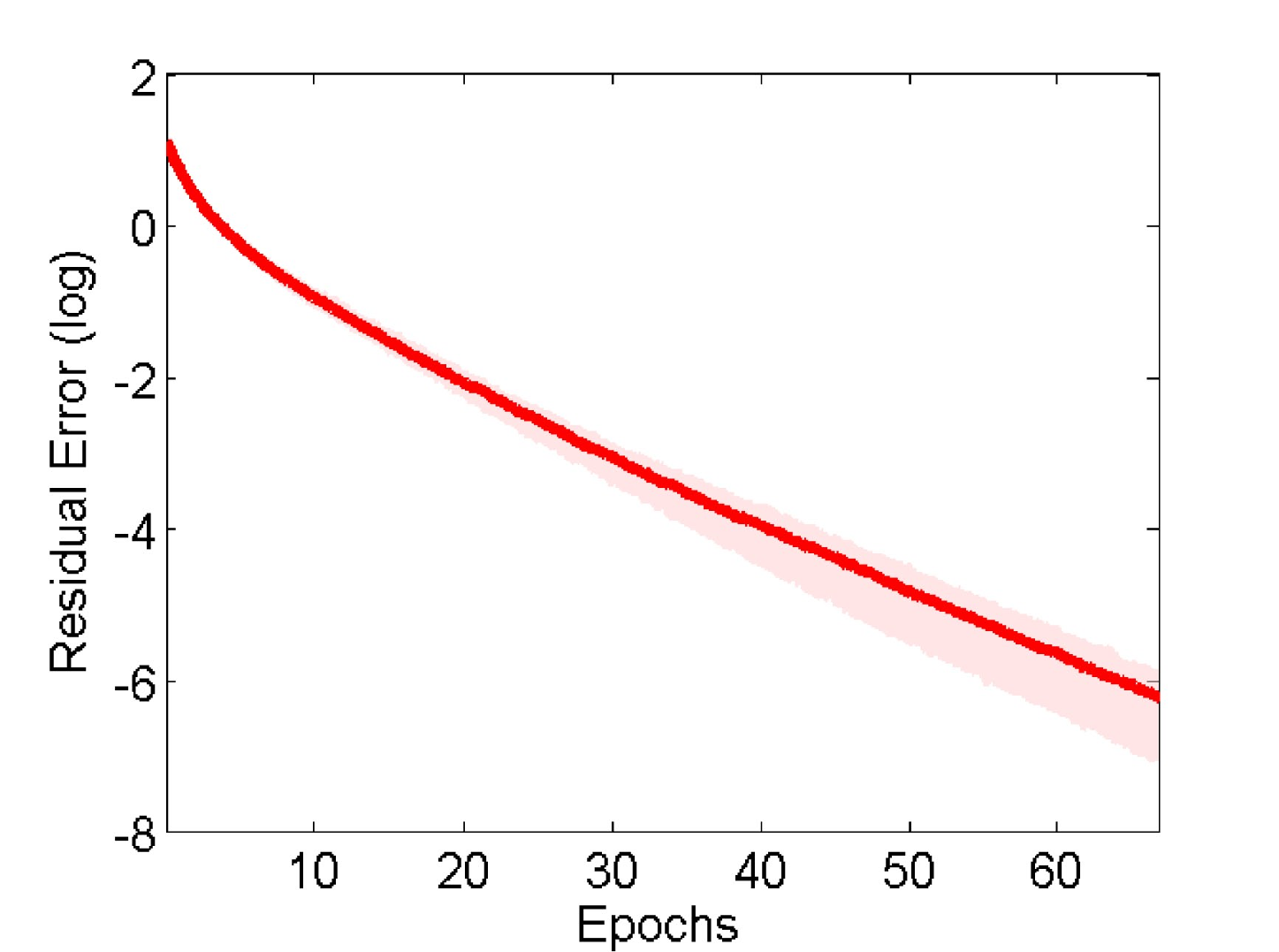} & \includegraphics[scale=0.35]{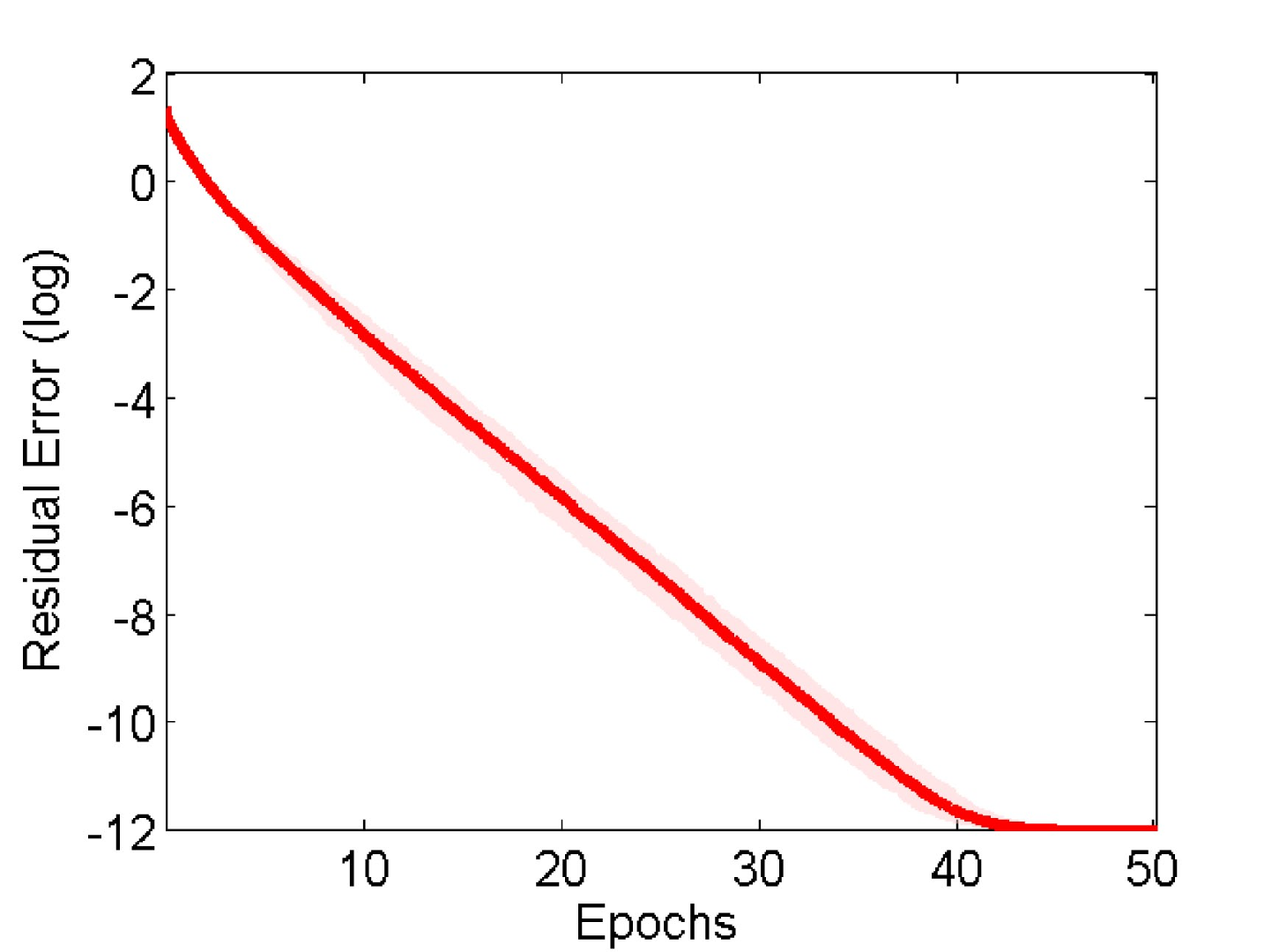} \\
(a) & (b) & (c) \\
\end{tabular}
\caption{ Residual error of the Kaczmarz Method per epoch: a) Median residual error of block method over $40$ trials for matrix $\mtx{A}$ not having an 
obtuse row paving, b) Median residual error of block method over $40$ trials for matrix $\mtx{A}$ using an 
obtuse row paving, c) Median residual error of simple method over $40$ trials for same matrix as in a). Shaded region spans across minimum and maximum
values over all trials and solid line denotes median value.\label{fig:bad}}
\end{center}
\end{figure}

With this definition we obtain the following result, whose
proof can be found in the appendix.

\begin{theorem}\label{thm:obtuse}
Let $\mtx{A}$ satisfy the assumptions of Theorem~\ref{thm:convergence} and in addition have an obtuse $(m', \beta')$ row paving of $\mtx{A}_{\leq}$.  Let $x_1, \ldots$
denote the iterates of Algorithm~\ref{alg2}.  Then using the notation of Theorem~\ref{thm:convergence},

\begin{align*}
\Expect [d(\vct{x}_j,S)^2]
	 \leq \left[1 - \frac{1}{L^2(\beta' m' + \beta m)}\right]^j d(\vct{x}_{0},S)^2.
\end{align*}
\end{theorem}

Note that row pavings of standardized matrices can be obtained readily, often by random partitions \cite{tropp2011improved,tropp2012user,NT12:Block-Kaczmarz}, whereas obtuse row pavings may be much
more challenging to obtain in general.  Of course, by default the trivial paving which assigns each set $\tau$ to a single row always admits an obtuse row paving.
We focus on Algorithm~\ref{alg} which paves only $\mtx{A}_=$, and leave further analysis of Algorithm~\ref{alg2} and 
constructions of obtuse row pavings for future work.

\section{Experiments}\label{sec:exps}

We use {\sc Matlab} to run some experiments using random matrices to test the convergence of the block Kaczmarz method applied to a 
system of equalities and inequalities. In each experiment, we create a random $500$ by $100$ matrix $\mtx{A}$ where each element is an 
independent standard normal random variable. Each entry is then divided by the norm of its row so that the matrix is standardized. 
The first 400 rows of matrix $\mtx{A}$ compose $\mtx{A}_{=}$, and the remaining 100 rows are set as inequalities of $\mtx{A}_{\leq}$ 
in the method described by~\eqref{eqn:ineq-system}. The experiments are run using the following procedure.
For each of 100 trials, 
\begin{enumerate}
\item Create matrix $\mtx{A}$ in the manner described above.
\item Create $\vct{x}_{\star}$ where each entry is selected independently from a standard normal distribution. Set $\vct{b} = \mtx{A} \vct{x}_{\star}$.
\item Pave submatrix $\mtx{A}_{=}$ into 16 blocks with 25 equalities per block by a random partitioning of the rows.
\item Set initial approximations $\vct{x}_{0}^{\text{block}} = \vct{x}_{0}^{\text{simp}} = \mtx{A}^*\vct{b}$.
\item Draw $q$ uniformly at random from $[0,1]$.
\begin{enumerate}
\item If $q \leq \frac{n_{e}}{\ n}$, choose block $\{1,...,m\}$ uniformly at random and update iterate $\vct{x}_{j}^{\text{block}}$ using~\eqref{eqn:block-algorithm}. 
(Note that the threshold $\frac{n_{e}}{\ n}$ is different than that given in the main algorithm and theorem, but it is easier to calculate and seems
to work fine in practice.)
\item Else, choose a row uniformly at random from $\{401,...,500\}$ and update iterate $\vct{x}_{j}^{\text{block}}$ using~\eqref{eqn:ineq-rule}.
\item Update iterate $\vct{x}_{j}^{\text{simp}}$ using~\eqref{eqn:ineq-rule}.
\end{enumerate}
\end{enumerate}
For both the simple and block algorithms, the median, minimum, and maximum values of the residual $\enormsq{e(\mtx{A}\vct{x}_{j}-\vct{b})}$ of the 100 trials are recorded for each iteration $j$.

Figure~\ref{fig:plots} compares the performance of the block Kaczmarz method used in this paper and the standard Kaczmarz method described by 
Leventhal and Lewis \cite{LL10:Randomized-Methods}. The plot in Figure~\ref{fig:plots} (a) compares convergence per iteration. 
As the block Kaczmarz method enforces multiple equalities per iteration, it is unsurprising that it performs better in this experiment.
Figure~\ref{fig:plots} (b) displays the convergence of the two methods per epoch. The block Kaczmarz algorithm has an epoch of $m + n_i$ iterations, and the
 standard Kaczmarz method has an epoch of size $n$. 
Here, to be fair we only count an iteration towards an epoch if the estimated solution $\vct{x}_{j} \neq \vct{x}_{j-1}$. 
Thus in the case where a chosen inequality 
is already satisfied for iteration $j$, this iteration does not count towards an epoch since no computation is being performed.  We noticed, however,
that whether or not we modified the count in this way, the behavior still produces results very similar to Figure~\ref{fig:plots}. Once again the experiments yielded faster 
convergence with the block Kaczmarz approach.  It is interesting to compare the results of Figure~\ref{fig:plots} (b) and those of Figure~\ref{fig:bad}
(b) and (c).  The per-epoch convergence of the methods and whether the block or standard appears faster varies slightly 
and depends on both the number of rows and columns.  In general, 
the per-epoch convergence rates are reasonably comparable, as the analysis suggests.
However, Figure~\ref{fig:plots} (c) compares the rate of convergence of the two algorithms by plotting the residual against the 
CPU time expended in the simulation. 
We believe that the ability to utilize efficient matrix--vector multiplication gives the method 
significantly improved convergence per second 
relative to the standard Kaczmarz algorithm, although other mechanisms may certainly be at work as well.

\begin{figure}[ht]
\begin{center}
\begin{tabular}{ccc}
\includegraphics[scale=0.35]{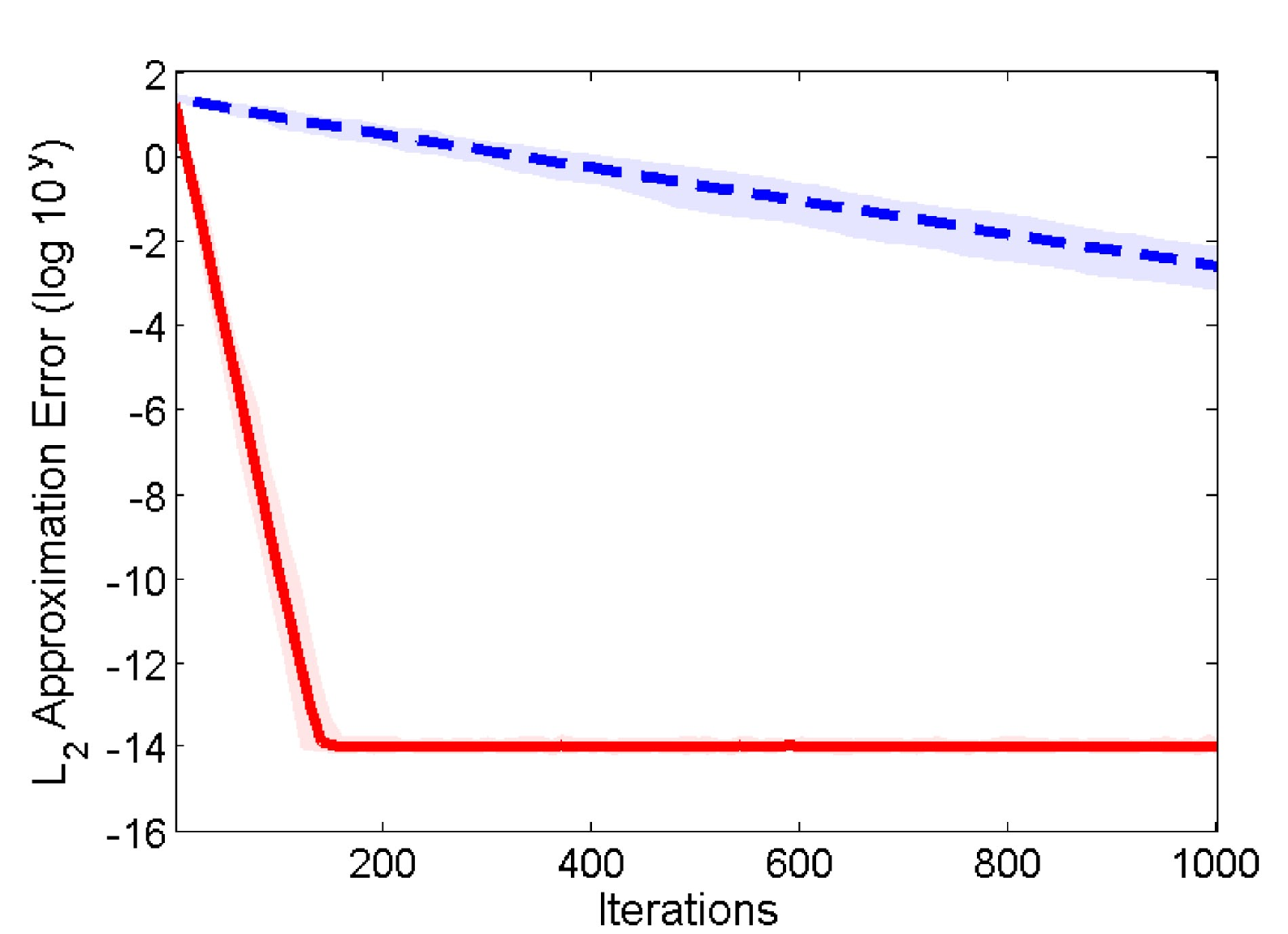} & \includegraphics[scale=0.35]{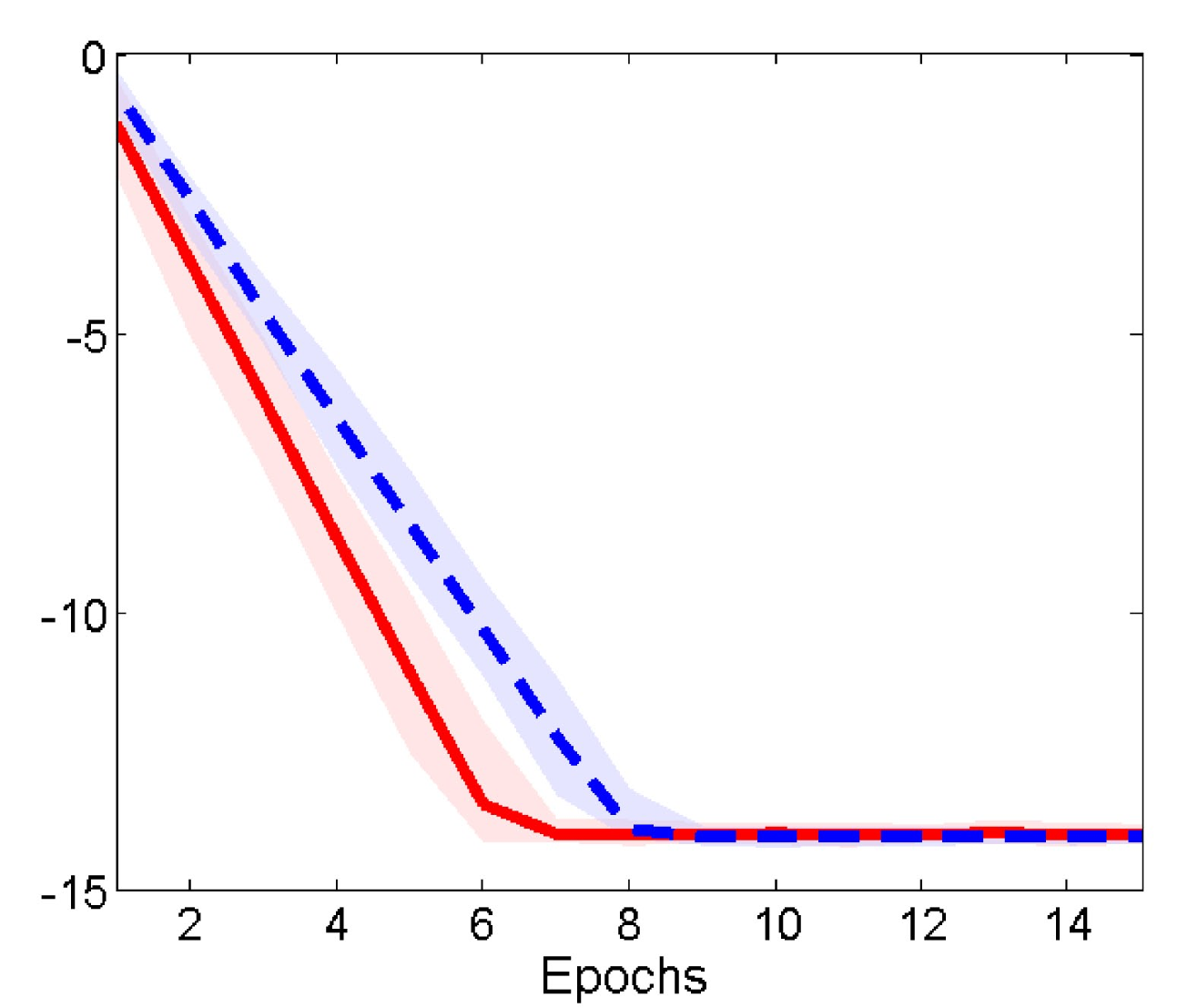} &  \includegraphics[scale=0.35]{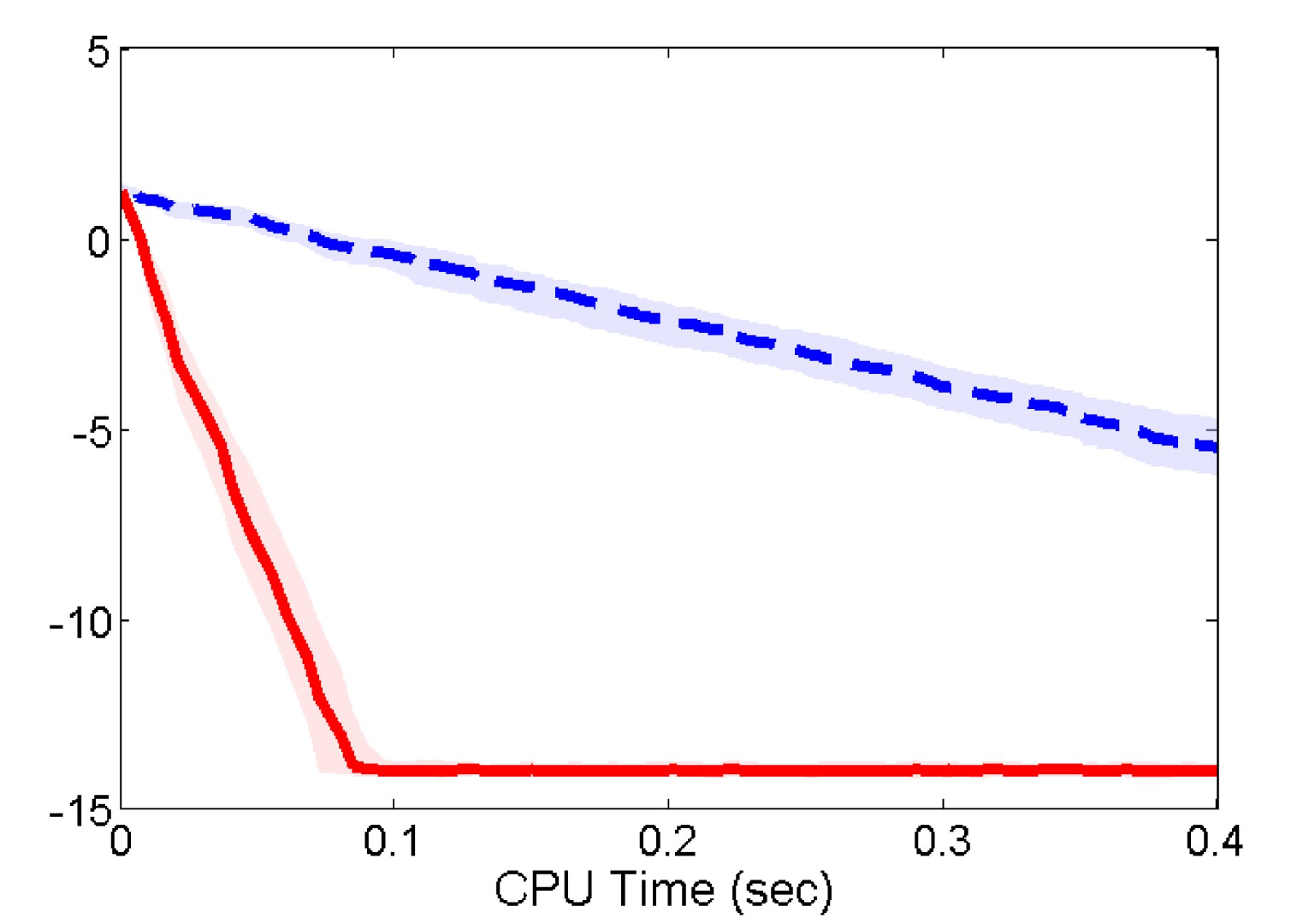} \\
(a) & (b) & (c) \\
\end{tabular}
\caption{ Residual error of the Block Kaczmarz Method (solid red) vs. Simple Kaczmarz Method (dashed blue) as a function of (a) Iterations, (b) Epochs, (c) CPU time.  Shaded region spans
from minimum to maximum value over $100$ trials; lines denote the median value.\label{fig:plots}}
\end{center}
\end{figure}

\section{Conclusion and Related Work}\label{sec:end} 
 The Kaczmarz algorithm was first proposed in
 \cite{Kac37:Angenaeherte-Aufloesung}. Kaczmarz demonstrated that the method
converged to the solution of linear system $\mtx{A}\vct{x} = \vct{b}$ for
square, non-singular matrix $\mtx{A}$.  Since then, the method has been
utilized in the context of computer tomography as the \textit{Algebraic
Reconstruction Technique} (ART) \cite{GBH70:Algebraic-Reconstruction,Byr08:Applied-Iterative,Nat01:Mathematics-Computerized,herman2009fundamentals}.
 Empirical results suggested that
randomized selection offered improved convergence over the cyclic
scheme \cite{HS78:Angles-Null,HM93:Algebraic-Reconstruction}. Strohmer and
Vershynin \cite{SV09:Randomized-Kaczmarz} were the first to prove an expected
linear convergence rate using a randomized Kaczmarz algorithm with specific
random control. This result was extended by Needell \cite{Nee10:Randomized-Kaczmarz} to apply to inconsistent systems, which shows a linear convergence
rate to within a fixed radius around the least-squares solution. 
Almost-sure convergence guarantees were recently proved by Chen and Powell \cite{CP12:Almost-Sure-Convergence}.
Zouzias and
Freris \cite{ZF13:Kaczmarz-Least-Squares} analyze a modified version of the
method in the inconsistent case, using a variant motivated by
Popa \cite{Pop99:Block-Projections-Algorithms} to reduce the residual and
thereby converge to the least squares solution.  Relaxation parameters
can also be introduced to obtain convergence to the least squares solution,
see e.g. \cite{whitney1967two,censor1983strong,tanabe1971projection,hanke1990acceleration},
and partially weighted sampling can lead to a tradeoff between convergence rate and radius \cite{NSW13:SGD-Kaczmarz}.
 Liu,
Wright, and Sridhar \cite{LWS14:Parallel-Kaczmarz} discuss applying a
parallelized variant of the randomized Kaczmarz method, demonstrating that the
convergence rate can be increased almost linearly by bounding the number of
processors by a multiple of the number of rows of $\mtx{A}$.

The block Kaczmarz updating method was introduced by Elfving \cite{Elf80:Block-Iterative-Methods} 
as a special case of the more general framework by Eggermont et.al. \cite{EHL81:Iterative-Algorithms}. 
The notion of using blocking in projection methods is certainly not new, and there is a large
amount of literature on these types of methods, see e.g. \cite{XZ02:Method-Alternating,Byr08:Applied-Iterative} and references therein. 
Needell and Tropp \cite{NT12:Block-Kaczmarz} provide the first analysis showing an
 expected linear convergence rate which depends on the
properties of the matrix $\mtx{A}$ and of the submatrices $\mtx{A_{\tau}}$ resulting
from the paving, connecting pavings and the block Kaczmarz scheme. 
The use of specialized blocks appears elsewhere, in particular, the works of
Popa use blocks with orthogonal rows that are beneficial for the
block Kaczmarz method \cite{Pop99:Block-Projections-Algorithms,Pop01:Fast-Kaczmarz-Kovarik,Pop04:Kaczmarz-Kovarik-Algorithm}. 
Needell, Zhao, and
Zouzias \cite{NZZ14:Block-Least-Squares} expand on the results
from \cite{NT12:Block-Kaczmarz} and \cite{ZF13:Kaczmarz-Least-Squares} to
demonstrate convergence to the least-squares solution for an inconsistent system
using the block Kaczmarz method. Again the block approach can yield faster
convergence than the simple method.

The Kaczmarz method was first applied to a system of equalities and inequalities by Leventhal and Lewis  \cite{LL10:Randomized-Methods},
who also consider polynomial constraints with the method. 
They give a linear convergence rate to the feasible solution space $S$, using $\fnormsq{\mtx{A}}$ and the Hoffman 
constant \cite{Hof52:Approximate-Solutions}. We apply the block Kaczmarz scheme to the system described in \cite{LL10:Randomized-Methods}, 
combining their result with that of Needell and Tropp \cite{NT12:Block-Kaczmarz} to acquire a completely generalized result.  We highlight several important complications
which arise when attempting to apply the block scheme to inequalities.  Nonetheless, whether a paving is used only partially or for
the complete system, significant reduction in computational time can be achieved.

\subsection{Future Work}
There are many interesting open problems related to the block Kazcmarz method and linear systems with inequalities.  It has been well observed
in the literature that selecting rows (or blocks) \textit{without replacement} rather than with replacement as in the theoretical results leads
to faster a convergence rate empirically \cite{RefWorks:533,NT12:Block-Kaczmarz}.  When selecting without replacement, independence
between iterations vanishes, making a theoretical analysis more challenging.  Secondly, it would be interesting to further investigate
the use of obtuse row pavings.  In systems with a large number of inequalities, the ability to pave the submatrix $\mtx{A}_\leq$ with an obtuse
row paving would lead to significantly faster convergence.  In that case, one may like to identify a more general geometric property about
the system that permits such pavings or an alternative formulation that offers convergence of the full block method.

\appendix

\section{Proof of Theorem~\ref{thm:obtuse}}
\begin{proof}
Fix an iteration $j$ of Algorithm~\ref{alg2}.  As in the proof of Theorem~\ref{thm:convergence}, if a block of
equalities is selected this iteration, then we again have~\eqref{eqs}.  So we next instead consider the case when a block of inequalities is selected, and call
this block $\tau'$, and its pruned subset $\sigma$.  Set $\vct{s} = P_S\vct{x}_{j-1}$, where again $P_S$ denotes the orthogonal projection onto the solution set $S$.
If we write $\tilde{H}_i = \{\vct{x} : \langle \vct{a_i}, \vct{x}\rangle \leq b_i\}$ and ${H}_i = \{\vct{x} : \langle \vct{a_i}, \vct{x}\rangle = b_i\}$, then by their definitions we have
$$
\vct{s} \in \cap_{i\in \sigma}\tilde{H}_i,\quad \vct{x}_{j-1} \in \cap_{i\in \sigma}\tilde{H}_i^c,\quad\text{and}\quad \vct{x}_j = 
P_{\cap_{i\in \sigma}{H}_i}\vct{x}_{j-1}.
$$
Then since $\sigma$ is part of an obtuse paving, the angle between $\vct{x}_j - \vct{x}_{j-1}$ and $\vct{s} -\vct{x}_{j-1}$ must be obtuse.  There thus exists a point
$\vct{t}$ on the line segment $L = \{\gamma\vct{x}_{j-1} + (1-\gamma)\vct{s} : 0 \leq \gamma \leq 1\}$ such that $\vct{x}_{j-1} - \vct{x}_j$ and $\vct{t} - \vct{x}_j$ are orthogonal (see Figure~\ref{draw5}).

\begin{figure}
\includegraphics[scale=0.4]{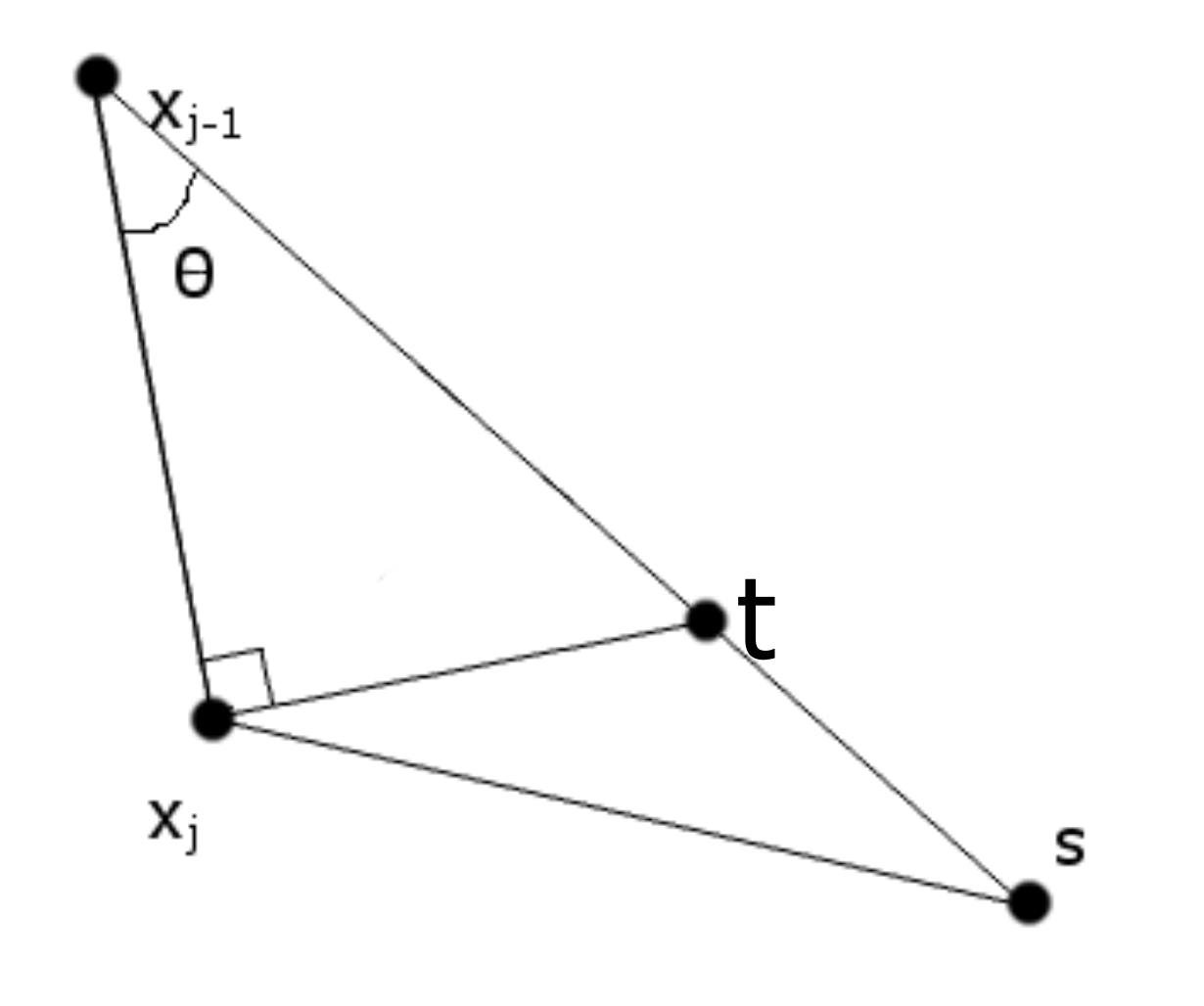}
\caption{Geometry of system.}\label{draw5}
\end{figure}

Now since $\vct{t} \in L$, we have $\enorm{\vct{t} - \vct{x}_{j-1}} \leq \enorm{\vct{x}_{j-1} - \vct{s}}$, and thus letting $\theta$ denote the angle between 
$\vct{x}_j - \vct{x}_{j-1}$ and $\vct{t} - \vct{x}_{j-1}$ (see Figure~\ref{draw5}), we have
\begin{align*}
\enorm{\vct{x}_j - \vct{x}_{j-1}} &\leq \enorm{\vct{x}_{j-1} - \vct{s}}\cdot\frac{\enorm{\vct{x}_j - \vct{x}_{j-1}} }{\enorm{\vct{t} - \vct{x}_{j-1}}}\\
&= \enorm{\vct{x}_{j-1} - \vct{s}}\cdot\cos\theta\\
&= \frac{\enorm{\vct{x}_j - \vct{x}_{j-1}}\cdot\enorm{\vct{s} - \vct{x}_{j-1}}\cdot\cos\theta}{\enorm{\vct{x}_j - \vct{x}_{j-1}}}\\
&= \frac{\langle \vct{s} - \vct{x}_{j-1}, \vct{x}_j - \vct{x}_{j-1}\rangle}{\enorm{\vct{x}_j - \vct{x}_{j-1}}}\\
&= \frac{-\langle \vct{x}_{j-1} - \vct{s}, \vct{x}_j - \vct{x}_{j-1}\rangle}{\enorm{\vct{x}_j - \vct{x}_{j-1}}}.
\end{align*}

Thus, we have that
$$
\langle \vct{x}_{j-1} - \vct{s}, \vct{x}_{j} - \vct{x}_{j-1}\rangle \leq -\enormsq{\vct{x}_j - \vct{x}_{j-1}}.
$$

By the definition of $\vct{x}_j$, this means that
\begin{equation}\label{star}
\langle \vct{x}_{j-1} - \vct{s}, \mtx{A}_{\sigma}^\pinv(\vct{b}_{\sigma} - \mtx{A}_{\sigma}\vct{x}_{j-1})\rangle \leq -
\enormsq{{\mtx{A}_{\sigma}^\pinv}(\vct{b}_{\sigma} - \mtx{A}_{\sigma}\vct{x}_{j-1})}.
\end{equation}

Using this along with the paving properties we see that
\begin{align*}
\enormsq{\vct{x}_{j} - \vct{s}} &= \enormsq{\vct{x}_{j-1} - \vct{s} + {\mtx{A}_{\sigma}^\pinv}(\vct{b}_{\sigma} - \mtx{A}_{\sigma}\vct{x}_{j-1})}\\
&= \enormsq{\vct{x}_{j-1} - \vct{s}} + 2\langle \vct{x}_{j-1} - \vct{s}, \mtx{A}_{\sigma}^\pinv(\vct{b}_{\sigma} 
- \mtx{A}_{\sigma}\vct{x}_{j-1})\rangle \\
&\;\;+ \enormsq{\mtx{A}_{\sigma}^\pinv(\vct{b}_{\sigma} - \mtx{A}_{\sigma}\vct{x}_{j-1})}\\
&\leq \enormsq{\vct{x}_{j-1} - \vct{s}} - \enormsq{\mtx{A}_{\sigma}^\pinv(\vct{b}_{\sigma} - \mtx{A}_{\sigma}\vct{x}_{j-1})}\\
&= d(\vct{x}_{j-1}, S)^2  - \enormsq{\mtx{A}_{\sigma}^\pinv(\vct{b}_{\sigma} - \mtx{A}_{\sigma}\vct{x}_{j-1})}\\
&\leq d(\vct{x}_{j-1}, S)^2  - \frac{1}{\beta'}\enormsq{\vct{b}_{\sigma} - \mtx{A}_{\sigma}\vct{x}_{j-1}}.\\
\end{align*}

Thus, taking expectation (over the choice of $\tau'$, conditioned on previous choices), yields
\begin{align*}
\Expect [d(\vct{x}_{j}, S)^2] &\leq \Expect \enormsq{\vct{x}_{j} - \vct{s}}\\
&\leq d(\vct{x}_{j-1}, S)^2 - \frac{1}{\beta'}\Expect\enormsq{\vct{b}_{\sigma} - \mtx{A}_{\sigma}\vct{x}_{j-1}}\\
&= d(\vct{x}_{j-1}, S)^2 - \frac{1}{\beta'}\Expect\enormsq{e(\vct{b}_{\tau'} - \mtx{A}_{\tau'}\vct{x}_{j-1})}\\
&= d(\vct{x}_{j-1}, S)^2 - \frac{1}{m'\beta'}\sum_{\tau'\in T'}\enormsq{e(\vct{b}_{\tau'} - \mtx{A}_{\tau'}\vct{x}_{j-1})}\\
&= d(\vct{x}_{j-1}, S)^2 - \frac{1}{m'\beta'}\enormsq{e(\vct{b}_{\leq} - \mtx{A}_{\leq}\vct{x}_{j-1})}.\\
\end{align*}

Combining this with~\eqref{eqs} and letting $E_=$ and $E_\leq$ denote the events that a block from $T$ and a block from $T'$ is selected, respectively, we have

\begin{align*}
\Expect \left[(d(\vct{x}_j,S)^2\right] &= p \cdot \Expect[ d(\vct{x}_{j},S)^2 | E_=] + (1-p) \cdot \Expect [ d(\vct{x}_{j},S)^2 | E_{\leq}] \\
&\leq p\left[d(\vct{x}_{j-1}, S)^2 - \frac{1}{\ \beta m} \sum_{i \in I_=} e(\mtx{A}_{=} \vct{x}_{j-1} - \vct{b}_=)_{i}^2\right] \\
&\;\;+ (1-p) \left[d(\vct{x}_{j-1}, S)^2 - \frac{1}{m'\beta'}\enormsq{e(\vct{b}_{\leq} - \mtx{A}_{\leq}\vct{x}_{j-1})}\right] \\
&= d(\vct{x}_{j-1}, S)^2 - p \cdot \frac{1}{\ \beta m} \sum_{i \in I_=} e(\mtx{A}_{=} \vct{x}_{j-1} - \vct{b}_=)_{i}^2\\
&\;\;- (1-p) \cdot \frac{1}{m'\beta'}\enormsq{e(\vct{b}_{\leq} - \mtx{A}_{\leq}\vct{x}_{j-1})}
\end{align*}

Since $p = \frac{\beta m}{\beta' m' + \beta m}$, we have $\frac{1-p}{\beta' m'} = \frac{1}{\beta' m'+\beta m}$ and we can simplify

\begin{align*}
\Expect \left[d(\vct{x}_j,S)^2\right] &\leq d(\vct{x}_{j-1}, S)^2 - \frac{1}{\beta' m' + \beta m} 
\Big[\sum_{i \in I_=} e(\mtx{A}_{=} \vct{x}_{j-1} - \vct{b}_=)_{i}^2 \\
&\;\;+  \enormsq{e(\vct{b}_{\leq} - \mtx{A}_{\leq}\vct{x}_{j-1})} \Big]  \\
&= d(\vct{x}_{j-1}, S)^2 - \frac{1}{\beta' m' + \beta m} \enormsq{e(\mtx{A} \vct{x}_{j-1} - \vct{b})} \\
&\leq d(\vct{x}_{j-1}, S)^2 - \frac{1}{L^2(\beta' m' + \beta m)} \cdot d(\vct{x}_{j-1},S)^2  \\
&= \left[1 - \frac{1}{L^2(\beta' m' + \beta m)}\right] d(\vct{x}_{j-1},S)^2,
\end{align*}
 where we have utilized the Hoffman bound~\eqref{hoffman} in the second inequality.

Iterating this relation along with independence of the random control completes the proof.

\end{proof}

%\bibliography{JB-v1}
%\bibliographystyle{spmpsci}
%\bibliographystyle{myalpha}

\end{document}